\documentclass[10pt, reqno]{amsart}
\usepackage{amsmath, amsthm, amscd, amsfonts, amssymb, graphicx, color}
\newcounter{minutes}
\divide\time by 60
\newcounter{hours}
\multiply\time by 60 \addtocounter{minutes}{-\time}

\newtheorem{theorem}{Theorem}[section]
\newtheorem{lemma}{Lemma}[section]
\newtheorem{corollary}{Corollary}[section]

\theoremstyle{remark}

\numberwithin{equation}{section}
\DeclareMathOperator{\IM}{Im}
\DeclareMathOperator{\RM}{Re}

\begin{document}
\setcounter{page}{1}

\title[Generalized Bessel functions]{Relations between the generalized Bessel functions and the Janowski class}

\author[S. Kanas, S. R. Mondal and A. D. Mohammed]{S. Kanas$^{1}$, S. R. Mondal$^2$  and A. D. Mohammed$^3$}
\address{$^{1}$ Department of Mathematical Analysis,
Faculty of Mathematics and Natural Sciences,
University of Rzeszow,
ul. St. Pigonia 1,
35-310 Rzeszow, Poland}
\email{skanas@ur.edu.pl}

\address{$^{2, 3}$ Department of Mathematics and Statistics, College of Science,
King Faisal University, Al-Hasa 31982, Hofuf, Saudi Arabia.}
\email{smondal@kfu.edu.sa, albhishi1900@gmail.com}



\begin{abstract}
We are interested in finding the sufficient conditions on $A$, $B$,
$\lambda$, $b$ and $c$  which ensure that the generalized Bessel
functions ${u}_{\lambda}:={u}_{\lambda,b,c}$  satisfies the
subordination ${u}_{\lambda}(z) \prec (1+Az)/ (1+Bz)$.  Also,
conditions for which ${u}_{\lambda}(z)$ to be Janowski convex, and
$z{u}'_{\lambda}(z)$ to be Janowski starlike in the unit disk
$\mathbb{D}=\{z \in \mathbb{C}: |z|<1\}$ are obtained.
\end{abstract}
 \maketitle

\section{Introduction}

We will denote by $\mathcal{A}$ the set of  functions $f$, analytic
in the open unit disk $\mathbb{D}=\{z: |z|<1\}$, and normalized by
the conditions  $f(0) = 0 = f'(0)-1$. If $f$ and $g$ are  analytic
in $\mathbb{D}$, then $f$ is \textit{subordinate} to $g$, written
$f\prec g$ (or $f(z) \prec g(z),$ $\ z \in \mathbb{D}$),  if there
is an analytic self-map $w$ of $\mathbb{D}$, satisfying $w(0)=0$ and
such that $f = g \circ w$. From now on, for $-1 \leq B < A \leq 1$\
the set $\mathcal{P}[A,B]$ denotes a family of functions $p(z) = 1+
c_{1}z + \cdots$, analytic  in $\mathbb{D}$ and satisfying
\[ p(z) \prec \frac{1+Az}{1+Bz}.\]
That family, known as the \textit{Janowski class of functions}
\cite{Janowski}, contains several other sets. For instance, if $0
\leq \beta <1$, then $ \mathcal{P}[1-2 \beta, -1]$ is the class of
functions $p(z)= 1+ c_{1}z + \cdots$ satisfying $ \RM p(z)
> \beta$ in $\mathbb{D}$ which, in the limiting case $\beta =0$, reduces
to the classical C\'{a}rath\`{e}odory class $\mathcal{P}$.

In relation to $\mathcal{P}[A,B]$ several subclasses of
$\mathcal{A}$ were defined, for example $\mathcal{S}^\ast[A, B]$,
which is called \textit{a class of Janowski starlike functions}
\cite{Janowski} and that consists of $ f \in \mathcal{A}$ satisfying
\[ z f'(z)/f(z)\in \mathcal{P}[A,B].\]
For $0 \leq \beta <1$, $ \mathcal{S}^\ast[1-2 \beta, -1]:=
\mathcal{S}^\ast(\beta)$ is the usual class of \textit{starlike
functions of order} $\beta$; $\mathcal{S}^\ast[1- \beta, 0]:=
\mathcal{S}^\ast_{\beta} = \{f \in \mathcal{A} : | z f'(z)/f(z) -1|
< 1- \beta \}$, and $\mathcal{S}^\ast[\beta, -
\beta]:=\mathcal{S}^\ast[\beta]= \{f \in \mathcal{A} : | z
f'(z)/f(z) -1| < \beta | z f'(z)/f(z) +1|\}$. These classes have
been studied, for example, in \cite{Ali-seeni-ijmms,
Ali-Chandra-aml}. A function $f \in \mathcal{A}$ is said to be
\textit{close-to-convex of order} $\beta$  if $\RM\left(
zf'(z)/{g(z)}\right)> \beta$ for some
 $g \in \mathcal{S}^\ast :=\mathcal{S}^\ast(0)$ \cite{Goodman-book, Miller-Mocanu-book}.

The second order differential equation of a real variable $x$ of a
form
\begin{equation}\label{Bessel_diff}
x^2u''+xu'+(x^2-\nu^2)u=0,
\end{equation}
is known as \textit{the Bessel differential equation}, where the
solutions of the Bessel equation yields the Bessel functions $J_\nu,
Y_\nu$ of the first and second kind , and $u=CJ_\nu(x)+DY_\nu(x)$
\cite[p. 217]{Olver_Lozier}. Here $C$ and $D$ are the arbitrary
constants and $\nu$ is an arbitrary complex number (the order of
Bessel function). The Bessel functions are named for Bessel however
Bernoulli is generally credited with being the first who introduced
the concept of Bessel's functions in 1732, when solved the hunging
chain problem. It is know that the Bessel function of the first kind
of order $\nu$ is defined by \cite{Olver_Lozier}
\begin{equation}\label{Bessel1}
J_\nu(x)=\sum\limits_{n=0}^\infty\frac{(-1)^n}{n!\Gamma(\nu+n+1)}\left(\frac{x}{2}\right)^{2n+\nu},
\quad x\in\mathbb{R}.
\end{equation}

Bessel's equation arises when solving the Laplace's and Helmholtz
equation and are therefore especially important for many problems of
wave propagation and static potentials. In finding the solution in
cylindrical coordinate systems, one obtains Bessel functions of
integer order ($\nu = n$), and in spherical problems one obtains
half-integer orders ($\nu = n+1/2$). There are several interesting
facts concerning the Bessel functions, in particular the connections
between Bessel functions and Legendre polynomials, hypergeometric
functions, the usual trigonometric functions and other.

The Bessel functions are valid  for complex arguments $x$, and an
important special case is that of a purely imaginary argument. In
this case, the solutions to the Bessel equation are called
\textit{the modified Bessel functions} or \textit{the hyperbolic
Bessel functions} of the first and second kind.  Several
applications have an impact of various generalizations and
modifications.

A second order differential equation which reduces to
\eqref{Bessel_diff} reads as follows
\begin{equation}\label{Bessel_diff_1}
x^2v''+bxv'+(cx^2-\nu^2+(1-b)\nu)u=0,
\end{equation}
$b,c,\nu \in \mathbb{R}$. A particular solution $v_\nu$ has the form
\begin{equation}\label{sol1}
v_\nu(x)=\sum\limits_{n=0}^\infty \frac{(-1)^n c^n
}{n!\Gamma(\nu+n+(b+1)/2}\left(\frac{x}{2}\right)^{2n+\nu},
\end{equation}
and is called \textit{the generalized Bessel function of the first
kind of order} $\nu$ \cite{Baricz2}. It is readily seen that for $b
= 1$ and $c = 1$, $v_\nu$ becomes  $J_\nu$.

Study of the geometric properties of some cases of Bessel functions,
like univalence, starlikeness and convexity were initiated in the
sixties by Brown \cite{Brown}, and also by Kreyszig and Todd
\cite{Kreyszig}, but a major contributions to the development of a
theory  in this direction was made by Baricz et al. see, for example
\cite{Baricz1} - \cite{Baricz-Szasz}.  Motivated by the importance
of the Bessel functions and the  results in the theory of univalent
functions we make a contribution to the subject, by obtaining some
necessary and sufficient conditions for Janowski starlikeness and
convexity of the generalized Bessel functions of the first kind.

For $b, \lambda, c \in \mathbb{C}$ and $\kappa$  such that $\kappa
=\lambda+(b+1)/2 \neq 0,-1,-2,-3, ...$ we denote by
$u_\lambda={u}_{\lambda,b,c}$ the \textit{normalized, generalized
Bessel function of of the first kind of order} $\lambda$ given by
the power series
\begin{align}\label{eqn:conf-hyp}
u_\lambda(z) =
2^\lambda\Gamma(\kappa+1)z^{-\lambda/2}J_\kappa(\sqrt{z})={}_0F_{1}\left(\kappa,\frac{-c}{4}z\right)
= \sum\limits_{k=0}^\infty \frac{(-1)^k
c^k}{4^k(\kappa)_k}\frac{z^k}{k!},
\end{align}
convergent for all $z$ on the complex plane. We note that
$u_\lambda(0)=1$, $u_\lambda$  is analytic in $\mathbb{D}$ and is a
solution of the differential equation
\begin{align}\label{eqn:kumar-hypr-ode}
4z^2 u''(z)+4\kappa zu'(z)+c z u(z)=0.
\end{align}
This normalized and generalized Bessel function also satisfy the
following recurrence relation \cite{Baricz1}
\begin{align}\label{eqn:kumar-hypr-recur-1}
4\kappa u'_\lambda(z)=-c u_{\lambda+1} (z),
\end{align}
which is an useful tool to study several geometric properties of
$u_\lambda$. There has been several papers, where geometric properties of
$u_\lambda$ such as on a close-to-convexity, starlikeness and convexity,
radius of starlikeness and convexity, were studied
\cite{Baricz1,Baricz2,Baricz-Pswamy,Baricz-Szasz,Selinger,Szasz-Kupan}.

In this paper we systematically study the properties of the
generalized Bessel function, specially Janowski convexity and
Janowski starlikeness  of that function.

In the section $\ref{section-3}$  of this paper, the sufficient
conditions on $A$, $B$, $c$, $\kappa$ are determined  that will
ensure that $u_\lambda$ satisfies the subordination $u_\lambda(z)
\prec (1+Az)/ (1+Bz)$. It is understood that a
computationally-intensive methodology is required to obtain the
results in this general framework. The benefits of such general
results are that by judicious choice of the parameters $A$ and $B$,
they give rise to several interesting applications, which include
extension of the results of previous works. Using this subordination
result, sufficient conditions are obtained for $(-4\kappa/c)
u'_\lambda \in \mathcal{P}[A, B]$, which next readily gives
conditions for $(-4\kappa/c)( u_\lambda-1)$ to be close-to-convex.
Section $\ref{section-4}$ gives emphasis to the investigation of
$u_\lambda$ to be Janowski convex as well as of $z u'_\lambda$ to be
Janowski starlike.

The following lemma is needed in the sequel.
\begin{lemma}\cite{miller-dif-sub, Miller-Mocanu-book}\label{lem:miller-mocanu-1}
Let $\Omega \subset \mathbb{C}$, and $\Psi : \mathbb{C}^2 \times
\mathbb{D} \to \mathbb{C}$ satisfy
\begin{equation}\label{l1.1}
 \Psi( i \rho , \sigma; z) \not \in \Omega,
\end{equation}
for  real $\rho$ and $\sigma$ such that $\sigma \leq -(1+\rho^2)/2$,
$z \in \mathbb{D}$. If $p$ is analytic in $\mathbb{D}$ with
$p(0)=1$, and $\Psi( p(z), z p'(z); z) \in \Omega$ for $ z \in
\mathbb{D}$, then $\RM  p(z) > 0$ in $\mathbb{D}$. In the case $\Psi
: \mathbb{C}^3 \times \mathbb{D} \to \mathbb{C}$, then the condition
\eqref{l1.1} is generalized to
\begin{equation}\label{l1.2}
 \Psi( i \rho , \sigma, \mu + i \nu; z) \not \in \Omega,
\end{equation}
where  $\rho, \sigma, \mu$ are real and such that  $\sigma+\mu \le
0$ and $ \sigma \le -(1+\rho^2)/2$.\end{lemma}

\section{Membership of the generalized Bessel functions of the Janowski class}\label{section-3}

In this section we shall discuss the problem of the membership of
the generalized Bessel function in the Janowski class. We find the
conditions under which $u_\lambda\in\mathcal{P}[A,B]$ and provide
several consequences of that fact.

\begin{theorem}\label{thm:-janw-cc}
Let $ -1 \le B <A \le 1$. Suppose $c, p, b  \in \mathbb{C}$  and $\kappa =p+(b+1)/2 \neq 0,-1,-2,-3 \cdots$,
satisfy
\begin{equation} \label{eqn:thm-janw-cc-0}
\RM(\kappa-1) \ge
\left\{\begin{array}{lcr}\dfrac{|c|}{4(1+A)}\left(\sqrt{2(1+A^2)}+(1-A)\right)&for&-1=B<A\le 3-2\sqrt{2},\\
\dfrac{|c|(1+A)}{8\sqrt{A}} \quad {and} \quad \RM(\kappa-1)\le \dfrac{|c|(1+A)}{2(1-A)}&for&B=-1, A>3-2\sqrt{2},\\
\dfrac{|c|(1+A)(1-B)^2}{4(A-B)(1+B)}-\dfrac{1+B}{(1-B)} &for&-1<B<0, \\
\dfrac{|c|(1+A)(1+B)}{4(A-B)}-\dfrac{1-B}{1+B}&for&B\ge0.\end{array}
\right.
\end{equation}
If $(1+B)u_\lambda \neq (1+A)$,  then $u_\lambda  \in
\mathcal{P}[A,B]$.
\end{theorem}

\begin{proof}
Define the analytic function $p : \mathbb{D} \to  \mathbb{C}$ by
\begin{align*}
p(z) = - \frac{(1-A)- (1-B)u_\lambda(z)}{ (1+A) - (1+B)
u_\lambda(z)},\quad p(0)=1.
\end{align*}
Then, a simple computation yields
\begin{align}\label{eqn-thm-1-1}
 u_\lambda(z)& = \frac{(1-A) + (1+A) p(z)}{(1-B) + (1+B) p(z)},\\ \label{eqn-thm-1-2}
u_\lambda'(z)& = \frac{ 2 (A-B) p'(z)}{((1-B) + (1+B) p(z))^2 },
\end{align}
and
\begin{align}\label{eqn-thm-1-3}
 u_\lambda''(z) =\frac{2 (A-B)[(1-B) + (1+B) p(z)] p''(z)
- 4 (1+B) (A-B) {p'}^2(z)}{( (1-B) + (1+B) p(z) )^3}.
\end{align}
Using the identities $(\ref{eqn-thm-1-1})$ -- $(\ref{eqn-thm-1-3})$,
the Bessel differential equation $(\ref{eqn:kumar-hypr-ode})$ can be
rewritten as
\begin{equation}\label{eqn:thm-1-ode} \nonumber\begin{array}{rcl}
z^2 p''(z) &-& \dfrac{ 2( 1+B)}{(1-B)+ (1+B)p(z)}  (zp'(z))^2 + \kappa zp'(z)\\
& & \\
&+& \dfrac{((1-B) + (1+B) p(z) ) ((1-A)+ (1+A)p(z))}{ 8 (A-B)} cz =
0.\end{array}
\end{equation}

Suppose $\Omega = \{0\}$, and  define $\Psi(r, s, t;z)$ by
\begin{align}\label{eqn:thm-1-ode-2}
\Psi(r, s, t;z) := t &- \frac{ 2( 1+B)}{(1-B)+ (1+B)r} s^2 + \kappa
s + \frac{((1-B) + (1+B)r)((1-A)+ (1+A)r)}{8(A-B)} cz.
\end{align}
The equation  $(\ref{eqn:thm-1-ode})$ yields that $ \Psi(p(z), z
p'(z), z^2 p''(z); z) \in \Omega.$  To ensure $\RM p(z) >0$ for $z
\in \mathbb{D}$ we will use the Lemma $\ref{lem:miller-mocanu-1}$.
Hence, it suffices to establish $ \RM \Psi( i\rho, \sigma, \mu+ i
\nu; z) \le 0 $ in $\mathbb{D}$  for  real $\rho, \sigma$ such that
$\sigma \leq -(1+\rho^2)/2$, and $\sigma+\mu \leq 0$. Applying those
inequalities we obtain
\begin{equation}\label{ineq:re-psi}
\begin{array}{rclcl}
\RM \Psi( i\rho, \sigma, \mu+ i \nu; z) & \leq
&-\dfrac{\RM(\kappa-1)}{2} (1+\rho^2)
&-&\dfrac{2(1-B^2)\sigma^2}{(1-B)^2+(1+B)^2\rho^2}\\
&& &+& \RM
 \dfrac{[(1-B)+(1+B)i\rho ][(1-A)+(1+A)i\rho]}{8(A-B)} cz\\
&\le &-\dfrac{\RM(\kappa-1)}{2} (1+\rho^2)
&-&\dfrac{(1-B^2)(1+\rho^2)^2}{2[(1-B)^2+(1+B)^2\rho^2]} \end{array}
\end{equation}
$$\hspace{7cm}+\
\dfrac{|(1-B)+(1+B)i\rho||(1-A)+(1+A)i\rho||c|}{8(A-B)}.$$

The proof will be divided into four cases. Consider first  $B=-1,
B<A\le 3-2\sqrt{2}$. The inequality \eqref{ineq:re-psi} reduces then
to the following
$$\begin{array}{rcl}\RM \Psi( i\rho, \sigma, \mu+ i \nu;
z)&\le&
-\dfrac{\RM(\kappa-1)(1+\rho^2)}{2}+\RM\dfrac{[(1-A)+(1+A)i\rho]cz}{4(1+A)}\\
&\le&-\dfrac{\RM(\kappa-1)(1+\rho^2)}{2}+\dfrac{|c|}{4(1+A)}[(1-A)+(1+A)|\rho|]\\
&=&-\dfrac{\RM(\kappa-1)}{2}\rho^2+\dfrac{|c|}{4}|\rho|+\dfrac{|c|(1-A)}{4(1+A)}-\dfrac{\RM(\kappa-1)}{2}\\
&=&-\dfrac{\RM(\kappa-1)}{2}\left(|\rho|-\dfrac{|c|}{4\RM(\kappa-1)}\right)^2+\dfrac{|c|^2}{32\RM(\kappa-1)}
+\dfrac{|c|(1-A)}{4(1+A)}-\dfrac{\RM(\kappa-1)}{2}\\
 &=:&G(\rho).\end{array}$$
A quadratic function $G$ takes nonpositive values for any $\rho$, if
$$\frac{|c|^2}{32\RM(\kappa-1)}+\frac{|c|(1-A)}{4(1+A)}-\frac{\RM(\kappa-1)}{2}\le
0.$$ The last inequality may be rewritten as
$$-\RM^2(\kappa-1)+\frac{|c|(1-A)}{2(1+A)}\RM(\kappa-1)+\frac{|c^2|}{16}\le
0,$$or
$$-\left(\RM(\kappa-1)-\frac{|c|(1-A)}{4(1+A)}\right)^2+\frac{|c|^2(1-A)^2}{16(1+A)^2}+\frac{|c|^2}{16}\le
0,$$that holds, if
$$\RM(\kappa-1) \ge
\frac{|c|}{4(1+A)}\left(\sqrt{(1-A)^2+(1+A)^2}+(1-A)\right),$$ which
reduces to the assumption. Therefore the assertion follows.

In the second case we consider $B=-1, A>3-2\sqrt{2}$.  According to
\eqref{ineq:re-psi}, we have
$$\begin{array}{rcl}\RM \Psi( i\rho, \sigma, \mu+ i \nu;
z)&\le&
-\dfrac{\RM(\kappa-1)(1+\rho^2)}{2}+\dfrac{|(1-A)+(1+A)i\rho||c|}{4(1+A)}\\
&=&-\dfrac{\RM(\kappa-1)(1+\rho^2)}{2}+\dfrac{|c|}{4(1+A)}\sqrt{(1-A)^2+(1+A)^2\rho^2}\\
&=:&H(\rho).\end{array}$$ We note that the function $H$ is even with
respect to  $\rho$, and
$$H(0) = \frac{|c|(1-A)}{4(1+A)}-\frac{\RM(\kappa-1)}{2},$$
that satisfies $H(0) \le 0$, if
\begin{equation}\label{k1}
\RM(\kappa-1)\ge \frac{|c|(1-A)}{2(1+A)}.
\end{equation}
Moreover $\lim\limits_{\rho\to\infty}H(\rho)=-\infty$, and
$$H'(\rho)=-\RM(\kappa-1)\rho+\frac{|c|(1+A)\rho}{4\sqrt{(1-A)^2+(1+A)^2\rho^2}},$$
with $H'(\rho)=0$ if and only if $\rho=0$ or
$$\rho^2_0=\frac{|c|^2}{16\RM^2(\kappa-1)}-\frac{(1-A)^2}{(1+A)^2}.$$
We observe that $\rho_0^2\ge 0$ by the inequality
$$\frac{|c|^2}{16\RM^2(\kappa-1)}\ge\frac{(1-A)^2}{(1+A)^2},$$
or
\begin{equation}\label{k2}
\RM(\kappa-1)\le \frac{|c|(1+A)}{4(1-A)}.
\end{equation}
Additionally
$$H''(\rho_0)=-\RM(\kappa-1)+\dfrac{16\RM^3(\kappa-1)(1-A)^2}{|c|^2(1+A)^2}\le
0,$$  in view of \eqref{k2}. Hence $H(\rho_0)=H_{\max}(\rho)$, and
$$H(\rho_0)=\dfrac{|c|^2}{32\RM(\kappa-1)}-\dfrac{\RM(\kappa-1)}{2}\left[1-\left(\dfrac{1-A}{1+A}\right)^2\right] \le 0$$
that holds if
\begin{equation}\label{k3}
\RM(\kappa-1)\ge \dfrac{|c|(1+A)}{8\sqrt{A}}.
\end{equation}
Since $$\dfrac{|c|(1+A)}{4(1-A)}\ge \dfrac{|c| (1+A)}{8\sqrt{A}} \ge
\dfrac{|c|(1-A)}{2(1+A)}$$ holds for $3-2\sqrt{2}\le A \le 1$, then
the conditions \eqref{k1}, \eqref{k2} and \eqref{k3} reduce to the
assumption \eqref{eqn:thm-janw-cc-0}. Therefore the assertion
follows.

Let now  $-1<B\le 0, A>B$. By the fact
$\frac{1-A}{1+A}<\frac{1-B}{1+B}$ we obtain
\begin{equation}\label{n2}
\begin{array}{rcl}|(1-B)+(1+B)i\rho||(1-A)+(1+A)i\rho|&=&(1+A)(1+B)\sqrt{\left(\frac{1-B}{1+B}\right)^2+\rho^2}
\sqrt{\left(\frac{1-A}{1+A}\right)^2+\rho^2}\\
&\le&(1+A)(1+B)\left[\left(\frac{1-B}{1+B}\right)^2+\rho^2\right].\end{array}
\end{equation}
Also, for $B\le 0$ we have $(1+B)/(1-B)\le 1$, therefore
$$\frac{1+\rho^2}{(1-B)^2+(1+B)^2\rho^2}=\frac{1}{(1-B)^2}\frac{1+\rho^2}{1+\left(\frac{1+B}{1-B}\right)^2\rho^2}\ge\frac{1}{(1-B)^2}$$
for any real $\rho$. Thus
\begin{equation*}\label{n1}\begin{array}{rcl}
\RM \Psi( i\rho, \sigma, \mu+ i \nu; z)&\le&
-\dfrac{\RM(\kappa-1)}{2} (1+\rho^2)
-\dfrac{(1+B)(1+\rho^2)}{2(1-B)}\\
&&\hspace{32mm}+\
 \dfrac{|c|(1+A)(1+B)}{8(A-B)}\left[\left(\dfrac{1-B}{1+B}\right)^2+\rho^2\right]\\
&=&\rho^2\left(-\dfrac{\RM(\kappa-1)}{2}-\dfrac{1+B}{2(1-B)}+\dfrac{|c|(1+A)(1+B)}{8(A-B)}\right)\\
&&\hspace{32mm}-\ \dfrac{\RM(\kappa-1)}{2}
-\dfrac{1+B}{2(1-B)}+\dfrac{|c|(1+A)(1-B)^2}{8(A-B)(1+B)}.\end{array}
\end{equation*} Since
for $B\le 0$
$$-\frac{\RM(\kappa-1)}{2}-\frac{1+B}{2(1-B)}+\frac{|c|(1+A)(1+B)}{8(A-B)}\le -\frac{\RM(\kappa-1)}{2}
-\frac{1+B}{2(1-B)}+\frac{|c|(1+A)(1-B)^2}{8(A-B)(1+B)},$$ and the
last expression is nonpositive in view of \eqref{eqn:thm-janw-cc-0}
then the assertion follows.

Finally, consider $0\le B<A\le 1$. In this case $\beta = (1-B)/(1+B)
\le 1$. Hence,  setting $t=\beta^2+\rho^2$ with $t\ge \beta^2$ and
using \eqref{n2}, we obtain from \eqref{ineq:re-psi}
$$
\RM \Psi( i\rho, \sigma, \mu+ i \nu; z)\le-\frac{\RM(\kappa-1)}{2}
(1-\beta^2+t) -\frac{\beta(1-\beta^2+t)^2}{2t}+
\frac{|c|(1+A)(1+B)}{8(A-B)}t$$
$$=t\left\{-\frac{\RM(\kappa-1)}{2}-\frac{\beta}{2}+\frac{|c|(1+A)(1+B)}{8(A-B)}\right\}
-\frac{\RM(\kappa-1)}{2}(1-\beta^2)-\frac{\beta(1-\beta^2)^2}{2t}-\beta(1-\beta^2)$$
that is nonpositive because of the inequality
\[\RM(\kappa-1)\ge \frac{|c|(1+A)(1+B)}{4(A-B)}-\frac{1-B}{1+B},\]
that is equivalent to the assumption \eqref{eqn:thm-janw-cc-0}.

Taking into account the above reasoning we see that $\Psi$ satisfies
the hypothesis of Lemma $\ref{lem:miller-mocanu-1}$, and thus $\RM\;
p(z) > 0$, that is,
\[
- \frac{ (1-A) - (1-B) u_\lambda(z)}{ (1+A) - (1+B) u_\lambda(z)}
\prec \frac{1+z}{1-z}.
\]
Hence there exists an analytic self-map  $w$ of $\mathbb{D}$ with
$w(0)=0$ such that
\[
- \frac{ (1-A) - (1-B) u_\lambda(z)}{ (1+A) - (1+B) u_\lambda(z)} =
\frac{1+w(z)}{1-w(z)},
\]
which implies that $u_\lambda(z) \prec (1+ A z)/(1+B z).$
\end{proof}

By the recurrence relation \eqref{eqn:kumar-hypr-recur-1} we have
$$\RM u_{\lambda+1}=\RM\left(\dfrac{-4\kappa}{c}u'_\lambda\right)$$
therefore as an immediate consequence of Theorem
$\ref{thm:-janw-cc}$ we obtain the following.

\begin{theorem}
Let $ -1 \le B <A \le 1$. Suppose $c, \lambda, b  \in \mathbb{C}$  and $\kappa =\lambda+(b+1)/2 \neq 0,-1,-2,-3 \cdots$,
satisfy
\begin{equation*}
\RM(\kappa) \ge
\left\{\begin{array}{lcr}\dfrac{|c|}{4(1+A)}\left(\sqrt{2(1+A^2)}+(1-A)\right)&for&-1=B<A\le 3-2\sqrt{2},\\
\dfrac{|c|(1+A)}{8\sqrt{A}}\quad {and} \quad \RM(\kappa)\le \dfrac{|c|(1+A)}{4(1-A)}&for&B=-1, A>3-2\sqrt{2},\\
\dfrac{|c|(1+A)(1-B)^2}{4(A-B)(1+B)}-\dfrac{1+B}{(1-B)} &for&-1<B<0, \\
\dfrac{|c|(1+A)(1+B)}{4(A-B)}-\dfrac{1-B}{1+B}&for&B\ge0.\end{array}
\right.
\end{equation*}
If $(1+B)u_{\lambda+1}(z) \neq (1+A)$,  then
$(-4\kappa/c)u'_\lambda(z) \in \mathcal{P}[A,B]$.
\end{theorem}

\section{Janowski convexity and starlikeness of the generalized Bessel functions}\label{section-4}

This section is devoted to the study of the Janowski convexity and
the Janowski starlikeness of the normalized and generalized Bessel
functions $z u_\lambda$. We proceed analogously to the proof of
Theorem \ref{thm:-janw-cc} applying modification of the  Bessel
differential equation \eqref{eqn:kumar-hypr-ode} and Lemma
\ref{lem:miller-mocanu-1}. An application of the Janowski convexity
and the  relation $(\ref{eqn:kumar-hypr-recur-1})$ yield conditions
for $z u_\lambda$ to be in $\mathcal{S}^*[A, B].$

 \begin{theorem}\label{thm:jan-convex}
 Let $ -1 \leq B < A \leq 1$ and $\lambda, b, c \in \mathbb{C}$ and $\kappa=\lambda+(b+1)/2 \neq 0, -1, -2 \ldots$.
Suppose that
 \begin{equation}\label{eqn:thm-jan-conv-1}
 \RM{\kappa} \ge  \frac{(\IM\kappa)^2}{2(2+A)}+ \frac{A}{2}+\frac{|c|}{2(A+1)}\quad \text{for}\quad -1=B<A \leq 1, \\
\end{equation}
or, for $-1<B <A \leq 1$
\begin{equation}\label{eqn:thm-jan-conv-2}
\frac{A-B-1}{1-B}+\frac{|c|(1-B)}{4(A-B)}\le \RM\kappa <
\frac{A-B+1}{1+B}-\frac{|c|(1+B)}{4(A-B)} ,\end{equation} and
\begin{equation}\label{eqn:thm-jan-conv-4}(B\IM\kappa)^2\le \left\{A-B+1-\frac{|c|(1+B)^2}{4(A-B)}-
(1+B)\RM\kappa\right\}\left\{(1-B)\RM\kappa
-\frac{|c|(1-B)^2}{4(A-B)}-A+B+1\right\},\end{equation} with
\begin{equation}\label{eqn:thm-jan-conv-3}
|c|< \frac{4(A-B)(1+B^2-AB)}{1-B^2}.
\end{equation}
If $(A-B)u'_\lambda(z) \neq (1+B) z u''_\lambda(z)$, $0 \notin
u'_\lambda(\mathbb{D})$ and $0 \notin u''_\lambda(\mathbb{D})$, then
\[1 + \frac{z u''_\lambda(z)}{u'_\lambda(z)} \prec \frac{1+Az}{1+Bz}.\]
\end{theorem}

\begin{proof}
Define an analytic function  $p : \mathbb{D} \to \mathbb{C}$ by
$$p(z): = \frac{ (A-B) u'_\lambda(z) + (1-B) z u''_\lambda(z)}{(A-B)
u'_\lambda(z) - (1+B) z u''_\lambda(z)},\quad p(0)=1.$$ Then
\begin{align}\label{eqn:thm-jan-conv-6}
 \frac{ z u''_\lambda(z)}{ u'_\lambda(z)}  =  \frac{(A-B) (p(z)-1)}{(1-B)+ (1+B)p(z)},
 \end{align}
and
\begin{align}\label{eqn:thm-jan-conv-6-1}\nonumber
\frac{z^2 u'''_\lambda(z)+zu''_\lambda(z)}{ zu''_\lambda(z)} - \frac{  zu''_\lambda(z) }{ u'_\lambda(z)}
&=  \frac{z p'(z)}{p(z)-1}-  \frac{(1+B) z p'(z)}{(1-B)+ (1+B)p(z)}\\
&= \frac{z p'(z) \left[{(1-B)+ (1+B)p(z)}-(1+B) (p(z)-1)\right]}{
(p(z)-1)[(1-B)+ (1+B)p(z)]}.
\end{align}
A rearrangement of $(\ref{eqn:thm-jan-conv-6-1})$ yields
\[ \frac{ z u'''_\lambda(z) }{ u''_\lambda(z)} = \frac{2 z p'(z)}{(p(z)-1)[(1-B)+
(1+B)p(z)]}-1+ \frac{zu''_\lambda(z)}{u'_\lambda(z)}. \] Thus,
\begin{align}\label{eqn:thm-jan-conv-7}\nonumber
&\left(\frac{ z u'''_\lambda (z)}{u''_\lambda(z)}\right)\; \left(\frac{zu''_\lambda(z)}{u'_\lambda(z)} \right)\\
&= \frac{2(A-B)(p(z)-1)z p'(z)}{(p(z)-1) \left((1-B)+
(1+B)p(z)\right)^2} - \frac{(A-B)(p(z)-1)}{(1-B)+ (1+B)p(z)}+
\frac{(A-B)^2(p(z)-1)^2}{\left((1-B)+ (1+B)p(z)\right)^2}.
\end{align}
Now a differentiation of $(\ref{eqn:kumar-hypr-ode})$  leads to
\begin{align*}
4z^2u'''_\lambda(z)+4(\kappa+1)zu''_\lambda(z)+czu'_\lambda(z) = 0,
\end{align*}
which gives if $u'_\lambda\neq 0$, $u''_\lambda\neq 0$
\begin{align}\label{eqn:thm-jan-conv-9}
\left(\frac{ z u'''_\lambda(z) }{ u''_\lambda(z)}\right)\; \left(
\frac{ zu''_\lambda(z) }{u'_\lambda(z)} \right) + (\kappa+1)\frac{
zu''_\lambda(z) }{u'_\lambda(z)}+\frac{c}{4}  z= 0.
\end{align}
Substituting $(\ref{eqn:thm-jan-conv-6})$ and
$(\ref{eqn:thm-jan-conv-7})$ into  $(\ref{eqn:thm-jan-conv-9})$  we
obtain
\begin{align*}
\frac{2(A-B)z p'(z)}{\left((1-B)+ (1+B)p(z)\right)^2}  +
\frac{(A-B)^2 (p(z)-1)^2}{\left((1-B)+ (1+B)p(z)\right)^2} +
\frac{\kappa(A-B) \left(p(z)-1\right)}{(1-B)+
(1+B)p(z)}+\frac{c}{4}z =0,
\end{align*}
or equivalently
\begin{align}\label{eqn:thm-jan-conv-10} \nonumber
z p'(z) + \frac{(A-B)}{2}\left(p(z)-1\right)^2 +
\frac{\kappa\left(p(z)-1\right)}{2} \left((1-B)+
(1+B)p(z)\right)+\frac{cz\left((1-B)+ (1+B)p(z)\right)^2}{8(A-B)}=0.
\end{align}
Set now
\begin{equation*}\begin{array}{ccc}\Psi(p(z),z p'(z);z) :=z p'(z) + \dfrac{(A-B)}{2}\left(p(z)-1\right)^2
&+&\dfrac{\kappa\left(p(z)-1\right)}{2} \left((1-B)+
(1+B)p(z)\right)\\
&+&\dfrac{cz\left((1-B)+(1+B)p(z)\right)^2}{8(A-B)}.\end{array}\end{equation*}
Then for $\rho \in \mathbb{R}$ and $ \sigma \leq - (1+ \rho^2)/2$ we
obtain
$$\begin{array}{rcl} \RM  \Psi(i \rho,\sigma;z)&=& \sigma +
\dfrac{(A-B)}{2} \RM(i \rho-1)^2+\RM \left(\dfrac{\kappa(i
\rho-1)}{2}\left((1-B)+(1+B)i \rho\right)\right)\\
&&\qquad\qquad\qquad\qquad\qquad\quad+\
\RM\left(\dfrac{cz\left((1-B)+(1+B)i
\rho\right)^2}{8(A-B)}\right)\\
&\le & -\dfrac{1+\rho^2}{2}+ \dfrac{(A-B)}{2}(1-\rho^2)+\RM
\left(\dfrac{\kappa}{2}
(-2Bi \rho-(1-B)-(1+B)\rho^2)\right)\\
&&\qquad\qquad\qquad\qquad\qquad\qquad\ +\ \dfrac{|c|\left((1-B)^2+(1+B)^2\rho^2\right)}{8(A-B)}\\
&=&-\rho^2\left\{\dfrac{A-B+1}{2}-\dfrac{(1+B)\RM\kappa}{2}-\dfrac{|c|(1+B)^2}{8(A-B)}\right\}
+(B\IM\kappa)\rho\\
&&\qquad\qquad\qquad\qquad\qquad\qquad+\ \dfrac{A-B-1}{2}-\dfrac{(1-B)\RM\kappa}{2}+\dfrac{|c|(1-B)^2}{8(A-B)}\\
&:=&Q(\rho).\end{array}$$ In order to get the contradiction we need
to show $Q(\rho) \le 0$ for $\rho\in \mathbb{R}$. We divide the
proof into two cases. Consider first the case $B=-1<A \leq 1$. Then
the function $Q$ becomes
\begin{align*}
Q (\rho)= -\frac{2+A}{2}\rho^2-(\IM\kappa)\rho+
\frac{A}{2}-\RM{\kappa}+\frac{|c|}{2(A+1)},
\end{align*}
that attains its maximum at $\rho_0= - \IM\kappa/(2+A)$, and
\begin{align*}
Q (\rho_0)= \frac{(\IM\kappa)^2}{2(2+A)}+
\frac{A}{2}-\RM{\kappa}+\frac{|c|}{2(A+1)}
\end{align*}
which is nonpositive by the assumption equivalent to
\eqref{eqn:thm-jan-conv-1}, that is
\begin{align*}
\RM{\kappa} \geq  \frac{(\IM\kappa)^2}{2(2+A)}+
\frac{A}{2}+\frac{|c|}{2(A+1)}.
\end{align*}

We now turn to the case $-1<B<A\leq 1$.  We rewrite $Q$ in the form
$$Q(\rho)=-P\rho^2+R\rho-S=-P\left\{\left(\rho-\frac{R}{2P}\right)^2+\frac{4PS-R^2}{4P^2}\right\},$$
where
$$P=\frac{A-B+1}{2}-\frac{(1+B)\RM\kappa}{2}-\frac{|c|(1+B)^2}{8(A-B)},$$
$$R=B\IM\kappa,\quad
S=\frac{(1-B)\RM\kappa}{2}-\frac{|c|(1-B)^2}{8(A-B)}-\frac{A-B-1}{2}.$$
The inequality $Q(\rho)\le 0$ holds for any real $\rho$, if $P>0,
S\ge 0$ and $R^2 \le 4PS$ or, equivalently
\begin{equation}\label{cond}\left\{\begin{array}{ccc}
         \dfrac{A-B+1}{1+B}-\dfrac{|c|(1+B)}{4(A-B)}& >&\RM\kappa, \\
         & & \\
         \dfrac{A-B-1}{1-B}+\dfrac{|c|(1-B)}{4(A-B)}& \le&\RM\kappa,
         \end{array}
\right.\end{equation} and
$$(B\IM\kappa)^2\le\left\{A-B+1-\frac{|c|(1+B)^2}{4(A-B)}-(1+B)\RM\kappa\right\}\left\{(1-B)\RM\kappa
-\frac{|c|(1-B)^2}{4(A-B)}-A+B+1\right\},$$ that holds by the
hypothesis \eqref{eqn:thm-jan-conv-2} and
\eqref{eqn:thm-jan-conv-4}. The inequalities \eqref{cond} can be
satisfied  only if
$$|c|< \frac{4(A-B)(1+B^2-AB)}{1-B^2}$$ that is equivalent to
\eqref{eqn:thm-jan-conv-3}. Therefore, in both cases the function
$\Psi$ satisfies the hypothesis of Lemma
$\ref{lem:miller-mocanu-1}$, and hence $\RM\; p(z) > 0$, or
equivalently
 \begin{align*}
\frac{ (A-B) u'_\lambda + (1-B) z u''_\lambda}{(A-B) u'_\lambda -
(1+B) z u''_\lambda} \prec \frac{1+z}{1-z}.
 \end{align*}
By definition of subordination, there exists an analytic self-map
$w$ of $\mathbb{D}$ with $w(0)=0$, and
\begin{align*}
\frac{ (A-B) u'_\lambda(z) + (1-B) z u''_\lambda(z)}{(A-B)
u'_\lambda(z) - (1+B) z u''_\lambda(z)} = \frac{1+w(z)}{1-w(z)},
 \end{align*}
that gives the equality
 \begin{align*}
1+ \frac{z u''_\lambda(z)}{u'_\lambda(z)} = \frac{1+Aw(z)}{1+Bw(z)}.
 \end{align*}
Hence
 \[1+ \frac{z u''_\lambda(z)}{u'_\lambda(z)} \prec \frac{1+Az}{1+Bz},\qedhere\]
 which is the desired conclusion.
\end{proof}

Based on the relation  $(\ref{eqn:kumar-hypr-recur-1})$ we also show
that
\begin{align*}
\frac{ z \; (z u'_\lambda(z)))'}{z u'_\lambda(z)} = 1+ \frac{
zu''_{\lambda-1}(z)}{u'_{\lambda-1}(z)}.
\end{align*}
Applying the above and  Theorem $\ref{thm:jan-convex}$, the
following result for $z u'_\lambda(z) \in \mathcal{S}^\ast[A,B]$
immediately follows.
\begin{theorem}
 Let $ -1 \leq B < A \leq 1$ and $\lambda, b, c \in \mathbb{C}$ and $\kappa=\lambda+(b+1)/2 \neq 0, -1, -2 \ldots$.
Suppose that
 \begin{equation}\label{3eqn:thm-jan-conv-1}
 \RM{\kappa} \ge  \frac{(\IM\kappa)^2}{2(2+A)}+ \frac{A}{2}+\frac{|c|}{2(A+1)}\quad \text{for}\quad -1=B<A \leq 1, \\
\end{equation}
or, for $-1<B <A \leq 1$
\begin{equation}\label{3eqn:thm-jan-conv-2}
\frac{A-B-1}{1-B}+\frac{|c|(1-B)}{4(A-B)}\le \RM\kappa <
\frac{A-B+1}{1+B}-\frac{|c|(1+B)}{4(A-B)} ,\end{equation} and
\begin{equation}\label{3eqn:thm-jan-conv-4}(B\IM\kappa)^2\le \left\{A-B+1-\frac{|c|(1+B)^2}{4(A-B)}-
(1+B)\RM\kappa\right\}\left\{(1-B)\RM\kappa
-\frac{|c|(1-B)^2}{4(A-B)}-A+B+1\right\},\end{equation} with
\begin{equation}\label{3eqn:thm-jan-conv-3}
|c|< \frac{4(A-B)(1+B^2-AB)}{1-B^2}.
\end{equation}
If  $(A-B)u'_\lambda(z) \neq (1+B) z u''_\lambda(z)$, $0 \notin
u'_\lambda(\mathbb{D})$ and $0 \notin u''_\lambda(\mathbb{D})$,
 then  $z u'_\lambda(z) \in \mathcal{S}^\ast[A,B]$.
\end{theorem}

In the special case $B=-1$ and $A=1-2\gamma$ we have from Theorem
$\ref{thm:jan-convex}$

 \begin{corollary}
 Let $\gamma \in [0,1)$ and $\lambda, b, c \in \mathbb{C}$ and $\kappa=\lambda+(b+1)/2 \neq 0, -1, -2
 \ldots$, and
 \begin{equation}
 \RM{\kappa} \ge  \frac{(\IM\kappa)^2}{2(3-2\gamma)}+
 \frac{1-2\gamma}{2}+\frac{|c|}{4(1-\gamma)}.\end{equation}
If  $0 \notin u'_\lambda(\mathbb{D})$  then
\[\RM\left(1 + \frac{z u''_\lambda(z)}{u'_\lambda(z)}\right)\ > \ \gamma.\]
\end{corollary}

\section*{Acknowledgements} This work was partially supported by the Centre for
Innovation and Transfer of Natural Sciences and Engineering
Knowledge, Faculty of Mathematics and Natural Sciences, University
of Rzeszow.\\

\end{document}